\documentclass[12pt]{amsart}
\usepackage{amsmath}
\usepackage{amsthm}
\usepackage{amssymb}
\usepackage{amsfonts}
\usepackage{enumerate}
\usepackage{extarrows}
\usepackage{geometry}
\usepackage{stmaryrd}
\usepackage{tikz-cd}
\usepackage[colorlinks,linkcolor=black,anchorcolor=black,citecolor=black]{hyperref}

\makeatletter
\def\subsection{\@startsection{subsection}{2}%
	\z@{.5\linespacing\@plus.7\linespacing}{.1\linespacing}%
	{\normalfont\bfseries}}
\makeatother

\newtheorem{theorem}{Theorem}[section]
\newtheorem{definition}[theorem]{Definition}
\newtheorem{corollary}[theorem]{Corollary}
\newtheorem{proposition}[theorem]{Proposition}

\theoremstyle{definition}

\newcommand{\cc}{\mathbb{C}}
\newcommand{\ddb}{\partial\bar{\partial}}
\newcommand{\db}{{\bar{\partial}}}
\newcommand{\dd}{\partial}

\newcommand{\cinf}{C^\infty}

\newcommand{\xt}{X_t}
\newcommand{\h}[2][DR]{H^{#2}_{#1}}

\title{Strongly Gauduchon Hyperbolicity and two other Types of Hyperbolicity}
\author{Yi Ma}

\begin{document}
	
	\maketitle
	
	\begin{abstract}
This paper proposes sG-hyperbolicity as a new tool for studying hyperbolicity on complex manifolds. It demonstrates that this notion leads to a wider class of divisorially hyperbolic manifolds compared to balanced hyperbolicity. We also introduce weakly p-Kähler hyperbolic structures and pluriclosed star split hyperbolic metrics as possible new avenues for exploration.
	\end{abstract}
	
	\section{Introduction}

	Hyperbolicity is an important concept in the theory of complex manifolds, characterizing their geometric and topological properties. In recent years, the study of hyperbolicity has attracted widespread attention in the fields of complex analytic, algebraic and differential geometries, and has achieved a series of important results. Classical notions such as Kähler hyperbolicity, Kobayashi hyperbolicity, and Brody hyperbolicity have been intensively studied. Meanwhile, new notions such as balanced hyperbolicity and divisorial hyperbolicity have been introduced and studied, providing new perspectives and tools for the study of complex manifolds.
	
	Let us first recall some notions of hyperbolicity.
	
	Let $X$ be a compact complex manifold with $\dim_\mathbb{C}X\ge 2$.
	\begin{enumerate}
		\item A form $\alpha$ on $(X, \omega)$ is said to be $\tilde{d}$(bounded) if the lift $\tilde{\alpha}$ of $\alpha$ to the universal cover $\tilde{X}$ od $X$ is $d$-exact with a $d$-potential bounded with respect to the lift $\tilde{\omega}$ of $\omega$.
		\item(\cite{gromov1991kahler}) X is said to be Kähler hyperbolic if $X$ admits a Kähler metric whose fundamental form  $\omega$ is $\tilde{d}$(bounded).

		\item(\cite{kobayashi1967invariant}) X is said to be Kobayashi hyperbolic if the Kobayashi pseudo-distance on X is a distance.
		\item(\cite{brody1978compact}) X is said to be Brody hyperbolic if all holomorphic maps $f: \mathbb{C} \to X$ are constant.
		\item(\cite{dan21}) X is said to be balanced hyperbolic if there is a balanced metric $\omega$ on $X$ such that $\omega^{n-1}$ is $\tilde{d}$(bounded) with respect to $\omega$.
		\item(\cite{dan21}) X is said to be divisorially hyperbolic if there is no holomorphic map $f: \mathbb{C}^{n-1} \to X$ such that $f$ is non-degenerate at some point and has subexponential growth in the sense of Definition \ref{div}.
	\end{enumerate}
	
	As for the relations among these hyperbolicities, it is known that a compact complex manifold is Kobayashi hyperbolic if and only if it is Brody hyperbolic. Besides, for a compact complex manifold, we have the following implications:
	
	\[
	\begin{tikzcd}
  & X \text{ is Kähler hyperbolic} \arrow[d, Rightarrow] \arrow[r, Rightarrow]
& X \text{ is Kobayashi/Brody hyperbolic} \arrow[d, Rightarrow] \\
& X \text{ is balanced hyperbolic} \arrow[r, Rightarrow]
& X \text{ is divisorially hyperbolic}
	\end{tikzcd}
	\]
	
This paper aims to further extend the research on hyperbolicity concepts. After reviewing existing notions of hyperbolicity and their mutual relationships, we introduce the notion of {\bf sG-hyperbolicity} and investigate its connection to divisorial hyperbolicity. A key property of sG-hyperbolicity is its deformation openness --- see Theorem \ref{defopen}. This property is known to hold for Kobayashi hyperbolic compact complex manifolds \cite{brody1978compact}, but it is still an open question whether it holds for K\"ahler hyperbolic, balanced hyperbolic and divisorially hyperbolic compact complex manifolds. We then construct examples of sG-hyperbolic manifolds that are not necessarily balanced hyperbolic. Finally, we propose new hyperbolicity notions, namely {\bf weakly p-Kähler hyperbolicity}, {\bf pluriclosed star split hyperbolicity}, and the relationship with divisorial hyperbolicity, laying the groundwork for further research.

\section{sG-Hyperbolic Manifolds}
\subsection{Definition and Properties}
Recall the definition of sG manifolds:
\begin{definition}[{\cite{popovici2013deformation}}]
	Let \( X \) be a compact complex manifold with \( \dim_{\mathbb{C}}X = n \).
	\begin{enumerate}
		\item A \( C^\infty \) positive definite \( (1,1) \)-form \( \omega \) on \( X \) is said to be a strongly Gauduchon (sG) metric if $\omega^{n-1}$ is the $(n-1,n-1)$-component of a real $d$-closed $\cinf$ $(2n-2)$-form $\Omega$.
		\item If \( X \) carries such a metric, \( X \) is said to be a strongly Gauduchon (sG) manifold.
	\end{enumerate}
\end{definition}

The first notion we introduce in this paper is contained in the
\begin{definition}
	Let $X$ be a compact complex manifold with $\dim_{\mathbb{C}}X=n$. A Hermitian metric $\omega$ on $X$ is said to be {\bf sG-hyperbolic} ({\bf strongly Gauduchon hyperbolic}) if there exists a real $d$-closed $(2n-2)$-form $\Omega$ on $X$ such that the $(n-1,n-1)$-component of $\Omega$ is $\omega_{n-1}:=\frac{\omega^{n-1}}{(n-1)!}$ and $\Omega$ is $\tilde{d}$(bounded) with respect to $\omega$.
	
	The manifold X is said to be {\bf sG-hyperbolic} if it carries an sG-hyperbolic metric.
\end{definition}
%

The first property we observe for these manifolds is given in
\begin{proposition}\label{product}
	The Cartesian product of sG-hyperbolic manifolds is sG-hyperbolic.
\end{proposition}
\begin{proof}
	Let $(X_1,\omega_1)$ and $(X_2,\omega_2)$ be sG-hyperbolic manifolds of respective dimensions $m$ and $n$, and let $\pi_1:\widetilde{X_1}\rightarrow X_1$ and $\pi_1:\widetilde{X_1}\rightarrow X_1$ be their universal covers. $\omega_1^{m-1}$(resp. $\omega_2^{n-1}$) is the $(m-1,m-1)$ (resp.$(n-1,n-1)$ ) component of $d$-closed real form $\Gamma_1$ (resp. $\Gamma_2$).
	
	We denote by $\omega=\sigma^*_1\omega_1+\sigma^*_2\omega_2$ the induced product metric on $X$. We have that $\omega^{n+m-1}=\tbinom{n+m-1}{m-1}\sigma^*_1\omega_1^{m-1}\wedge\sigma^*_2\omega_2^n+\tbinom{n+m-1}{n-1}\sigma^*_1\omega_1^{m}\wedge\sigma^*_2\omega_2^{n-1}$ is the $(n+m-1,n+m-1)$-component of $$\Gamma=\tbinom{n+m-1}{m-1}\sigma^*_1\Gamma_1\wedge\sigma^*_2\omega_2^n+\tbinom{n+m-1}{n-1}\sigma^*_1\omega_1^{m}\wedge\sigma^*_2\Gamma_2,$$
	which is a $d$-closed real $(n+m-1)$-form. Therefore $\omega$ is a strongly Gauduchon metric.
	
	
	Besides, we know that
	\begin{align*}
		\pi^*\Gamma&=\tbinom{n+m-1}{m-1}\pi^*\sigma^*_1\Gamma_1\wedge\pi^*\sigma^*_2\omega_2^n+\tbinom{n+m-1}{n-1}\pi^*\sigma^*_1\omega_1^{m}\wedge\pi^*\sigma^*_2\Gamma_2\\
		&=\tbinom{n+m-1}{m-1}\tilde\sigma^*_1(\pi_1^*\Gamma_1)\wedge\sigma^*_2\omega_2^n+\tbinom{n+m-1}{n-1}\sigma^*_1\omega_1^{m}\wedge\tilde\sigma^*_2(\pi_2^*\Gamma_2)\\&=d[\tbinom{n+m-1}{m-1}\tilde\sigma^*_1\Theta_1\wedge\sigma^*_2\omega_2^n+\tbinom{n+m-1}{n-1}\sigma^*_1\omega_1^{m}\wedge\tilde\sigma^*_2\Theta_2].
	\end{align*}
	Hence $\Gamma$ is $\tilde{d}$(bounded) on $X_1\times X_2$, i.e. $X_1\times X_2$ is sG-hyperbolic.
\end{proof}
We have the deformation openness of the sG-hyperbolicity.
\begin{theorem}\label{defopen}
	Let \(\pi: X \rightarrow B\) be a holomorphic family of compact complex manifolds \(X_t := \pi^{-1}(t)\), with \(t \in B\). Fix an arbitrary reference point \(0 \in B\). If the fibre \(X_0\) is an sG-hyperbolic manifold, then, for all \(t \in B\) sufficiently close to 0, the fibre \(X_t\) is again an sG-hyperbolic manifold.
	
\end{theorem}
\begin{proof}
	
	Let \(\omega_0\) be an sG-hyperbolic metric on \(X_0\). By the definition of sG-hyperbolic metric, there exists a \(d\)-closed \(\tilde{d}\text{(bounded)}\) \((2n-2)\)-form \(\Omega\) such that \(\omega_0^{n-1}\) is the \((n-1,n-1)\)-component of \(\Omega\) on \(X_0\). 
	
	The \((n-1,n-1)\)-component \(\Omega^{n-1,n-1}\) of $\Omega$ with respect to the complex structure of \(X_t\) is positive definite for \(t \in B\) sufficiently close to \(0\). By Lemma (\cite{10.1007/BF02392356}, (4.8)), there exists a metric \(\omega_t\) such that \(\omega_t^{n-1} = \Omega_t^{n-1, n-1}\). Because of the compactness of the $\cinf$ manofold $X$ underlying the fibres $\xt$ and the continuity of $\omega_t$ with respect to $t$, $\Omega$ is $\tilde{d}$(bounded) with respect to $\omega_t$ for \(t \in B\) sufficiently close to \(0\). Hence \(X_t\) is again an sG-hyperbolic manifold.\end{proof}

Let us now finish recalling the definition of a divisorially hyperbolic manifold by recalling the definition of subexponential growth for entire holomorphic maps $f:\cc^p\rightarrow (X,\omega)$. We denote the open ball (resp. sphere) of radius $r$ centered at 0 in $\cc^{p}$ by $B_r$ (resp. $S_r$). Let $\star_\omega$ denote the Hodge star operator induced by a Hermitian metric $\omega$, and let $\tau(z):=|z|^{2}$ be the squared Euclidean norm on $\cc^{p}$.
\begin{definition}[\cite{dan21}]\label{div}
	
	Let $X$ be a compact complex manifold with $\dim_\mathbb{C}X\ge 2$. For $0< p\le n-1$, we say that a holomorphic map $f:\cc^{p}\rightarrow X$ has {\bf subexponential growth} if the following two conditions are satisfied:
\begin{enumerate}
	\item There exist constants $C_1>0$ and $r_0>0$ such that
	$$ \int_{S_{t}}|d \tau|_{f^{\star} \omega} d \sigma_{\omega, f, t} \leq C_{1} t \operatorname{Vol}_{\omega, f}\left(B_{t}\right), \quad t>r_{0},$$
	where  $d \sigma_{\omega, f,t}=\left.{ }({\star}_{f^{*} \omega}\left(\frac{d \tau}{|d \tau|_{f^{*} \omega}{ }}\right)\right)|_{s_{t}}$ . 
	
	\item For every constant $C>0$, we have:
	$$\varlimsup_{b\rightarrow+\infty}(\frac{b}{C}-\log F(b))=+\infty,$$
	where
	$$F(b):=\int_{0}^{b} \operatorname{Vol}_{\omega, f}\left(B_{t}\right) d t=\int_{0}^{b}\left(\int_{B_{t}} f^{*} \omega_{n-1}\right) d t, \quad b>0.$$
\end{enumerate}
\end{definition}
\begin{theorem}
	Every sG-hyperbolic manifold is divisorially hyperbolic.
\end{theorem}
\begin{proof}
	Let $X$ be a compact complex manifold of dimension $n$ and let $\omega$ be an sG-hyperbolic metric on $X$. Suppose there exists a holomorphic map $f:\cc^{n-1}\rightarrow X$ non-degenerate at some point that has subexponential growth.

	Let $\pi: \tilde{X} \longrightarrow X$ be the universal cover of $X$. There exists a  $\pi^{*}\omega$-bounded  $(2n-3)$-form  $\Gamma$  on  $\tilde{X}$ such that  $d \Gamma=\pi^{*} \Omega=\pi^{*}\left(\Omega^{n,n-2}+\omega_{n-1}+\Omega^{n-2, n}\right)$ .
	
	Then $\tilde{f}^{*} \Gamma$ is $f^*\omega$-bounded because:
	\begin{align*}
		\left|\tilde{f}^{*} \Gamma\left(v_{1}, \cdots v_{2 n-3}\right)\right| & =\left|\Gamma\left(f_{*} v_{1}, \cdots, f_{*} v_{2 n-3}\right)\right| \\
		& \leqslant C\left|\tilde{f}_{*} v_{1}\right|_{\pi^{*} \omega} \cdots\left|\tilde f_{*} v_{2 n-3}\right|_{\pi^{*} \omega} \\
		& =C\left|v_{1}\right|_{f^{*} \omega} \cdots\left|v_{2 n-3}\right|_{f^{*} \omega}
	\end{align*}
	for any tangent vectors $v_{1}, \cdots, v_{2 n-3}$ in $\mathbb{C}^{n-1}$.
	
	Now, we have
	$$
	\begin{aligned}
		\text { Vol }_{\omega, f}\left(B_{r}\right) & =\int_{B r} f^{*} \omega_{n-1}=\int_{B r} f^{*} \Omega \\
		& =\int_{B r} d\left(\tilde{f}^{*}\Gamma\right) \leqslant C \int_{S_{r}} d \sigma_{\omega, f, r}.
	\end{aligned}$$

	By the H\"older inequality, we have
	
	$$\int_{s_{r}} \frac{1}{|d \tau|_{f^{*} \omega}} d \sigma_{\omega, f, r} \cdot \int_{S_{r}} |d \tau|_{f^{*} w} d \sigma_{\omega, f, r} \geqslant\left(\int_{S_{r}} d \sigma_{\omega, f,r}\right)^{2}.$$
	
	We have  $d \tau=2 r d r$ . Let  $d \mu_{\omega, f, r}$  be the measure on  $S_{r}$  such that
	
	$$d \mu_{\omega, f, r} \wedge \frac{(d \tau)|_{S_r}}{2 r}=\left.\left(f^{*} \omega_{n-1}\right)\right|_{S_r} \text {. }$$
	
	Hence, we have  $$\frac{1}{2 r} d \mu_{\omega,f , r}=\frac{1}{|d \tau|_{f^{*}\omega} } d \sigma_{\omega, f, r}. $$
	
	Then we get:
	$$\begin{array}{l}
		\text {Vol}_{\omega, f}\left(B_{r}\right)=\int_{0}^{r}\left(\int_{S_t} \frac{1}{| d \tau |_{f^*\omega} } d \sigma_{\omega, f, t}\right) d \tau \\
		\geqslant \int_{0}^{r} \frac{\left(\int_{S_{r}} d \sigma_{\omega, f,r}\right)^{2}}{\int_{S_t}\left|d \tau\right|_{f^{*} \omega} d \sigma_{\omega, f, t}} 2 t d t \\
		
		\geqslant \frac{2}{C^{2}} \int_{0}^{r} \frac{\left(\text {Vol}_{\omega, f}\left(B_{t}\right)\right)^{2}}{\int_{S_{t}}|d \tau|_{f^{*} \omega} d{\sigma_{\omega, t}}}td t \\
		\geqslant \frac{2}{C_{1} C^{2}} \int_{0}^{r} \text {Vol}_{\omega, f}\left(B_{t}\right) d t
	\end{array}$$
	for $r$ big enough.
	
	This is to say
	$$
	F^{\prime}(r) \geqslant \frac{2}{C_{1} C^{2}} F(r) \text { ,}$$
	for $r$ big enough.
	 
	Hence, we get$$	(\log F(r))^{\prime} \geqslant \frac{2}{C_{1} C^{2}}.$$
	Finally, we get	$$\log F\left(r_{1}\right)-\frac{2}{C_{1} C^{2}} r_{1}+\frac{2}{C_{1} C^{2}} r_{2} \geqslant \log F\left(r_{2}\right).
	$$
	By (ii) of Definition \ref{div}, we have $F(r_2)=0$ for $r_2$ big enough, which contradicts our assumption.
\end{proof}

\begin{theorem}
	Let $X$ be a compact complex manifold with $\dim_{\mathbb{C}}X=n$. Let $\pi:\tilde{X}\rightarrow X$ be the universal cover of $X$.
	If $X$ is sG-hyperbolic, then there exists no
	non-zero $d$-closed positive (1,1)-current $\tilde{T}\ge 0$
	on $\tilde{X}$ such that $\tilde{T}$ is of $L^1_{\pi^*\omega}$.
\end{theorem}

\begin{proof}
	By the definition of sG-hyperbolic manifold, there exists a \( d \)-closed \( (2n-2) \)-form \( \Omega \) on \( X \) where the \( (n-1,n-1) \) component is \( \omega^{n-1} \), and there exists a \( L^\infty_{\pi^*\omega} \)-form \( \Gamma \) of degree \( (2n-3) \) on \( \tilde{X} \) such that \( \pi^*\Omega = d\Gamma \).
	
	For a \( d \)-closed \( (1,1) \)-current \( \tilde{T}\) of \( L^1_{\pi^*\omega} \) on \( \tilde{X} \), \( \tilde{T} \wedge \Gamma \) is again  of \( L^1_{\pi^*\omega} \). Hence also \( d(\tilde{T} \wedge \Gamma) \). Now we have:
	
	\[
	\int_{\tilde{X}} \tilde{T} \wedge \pi^*\omega^{n-1} = \int_{\tilde{X}} \tilde{T} \wedge \pi^*\Omega = \int_{\tilde{X}} d(\tilde{T} \wedge \Gamma) = 0.
	\]
	
	Therefore, there exists no non-zero $d$-closed positive (1,1)-current $\tilde{T}\ge 0$
	on $\tilde{X}$ such that $\tilde{T}$ is of $L^1_{\pi^*\omega}$.
	
\end{proof}

\subsection{Example}

 To further explore the properties of sG-hyperbolic manifolds, we will now look for some examples of sG-hyperbolic manifolds that do not necessarily belong to the category of balanced hyperbolic manifolds.

\vspace{2ex}(a)For convenience, let us name a class of Hermitian metrics as follows:
\begin{definition}
	Let $X$ be a compact complex manifold with $\dim_\mathbb{C}X\ge 2$. A $C^\infty$ positive definite $(1, 1)$-form $\omega$ on $X$ is said to be a {\bf degenerate sG metric} if $\omega^{n-1}=\dd\alpha+\db\beta$ for some $\alpha\in\cinf_{n-2,n-1}(X,\cc)$, $\beta\in\cinf_{n-1,n-2}(X,\cc)$.
	If $X$ carries such a metric, X is said to be a {\bf degenerate sG manifold}.
\end{definition}
In other words, we require $\omega^{n-1}$ to define the zero Aeppli cohomology class (i.e. to be Aeppli-exact). Recall the definitions of Bott-Chern and Aeppli cohomology groups of bidegree $(p,q)$:
\begin{equation*}
	H^{p,q}_{BC}(X, \mathbb{C}) = \frac{\ker(\partial : C^{p,q}(X) \to C^{p+1,q}(X)) \cap \ker(\bar{\partial} : C^{p,q}(X) \to C^{p,q+1}(X))}
	{\mathrm{Im}(\partial \bar{\partial} : C^{p-1,q-1}(X) \to C^{p,q}(X))},
\end{equation*}
\begin{equation*}
	H^{p,q}_{A}(X, \mathbb{C}) = \frac{\ker(\ddb : C^{p,q}(X) \to C^{p+1,q+1}(X))}
	{\mathrm{Im}(\partial : C^{p-1,q}(X) \to C^{p,q}(X)) + \mathrm{Im}(\bar{\partial} : C^{p,q-1}(X) \to C^{p,q}(X))}.
\end{equation*}

It is obvious that degenerate sG metrics are strongly Gauduchon metrics.

Let us characterize degenerate sG manifolds, with contributions from \cite{dan22} and \cite{alessandrini2018forms}.
\begin{theorem}\label{degsg}
	Let $X$ be a compact complex manifold with $\dim_{\mathbb{C}}X=n$.
	\begin{enumerate}
		\item Let $\omega$ be a Hermitian metric on $X$, $\omega$ is degenerate sG if and only if there exists a $\cinf$ $d$-exact $(2n-2)$-form $\Omega$ on $X$ whose $(n-1,n-1)$-component is $\omega^{n-1}$.
		\item Let $\omega$ be a Gauduchon metric on $X$, $\omega$ is degenerate sG if and only if $\h[BC]{1,1}(X,\cc)\wedge[\omega^{n-1}]_A=0$, where $\h[BC]{1,1}(X,\cc)$ is the Bott-Chern cohomology group of bidegree $(1,1)$ and $[\omega^{n-1}]_A$ is the Aeppli cohomology class determined by $\omega^{n-1}$.
		\item $X$ is a degenerate sG manifold if and only if there exists no non-zero $d$-closed bidegree $(1, 1)$-current $T\ge0$ on $X$.
	\end{enumerate}
\end{theorem}
\begin{proof}
	(1) Let $\Omega$ be the $d$-exact $(2n-2)$-form mentioned in (1), then the $(n-1,n-1)$-component of $\Omega$ is obviously in $\text{Im} \partial + \text{Im}\db$. Conversely, if we have $\omega^{n-1}=\dd\alpha+\db\beta$ for some $\alpha\in\cinf_{n-2,n-1}(X,\cc)$ and some $\beta\in\cinf_{n-1,n-2}(X,\cc)$, then $d(\alpha+\beta)$ is a $d$-exact $(2n-2)$-form $\Omega$ on $X$ whose $(n-1,n-1)$-component is $\omega^{n-1}$.
	
	(2)
	Let us check that $[\omega^{n-1}]_A \wedge \cdot : H^{1,1}_{BC}(X, \mathbb{C}) \rightarrow H^{n,n}_A(X,\cc)$ is well-defined first.
	
	Because the Hermitian metric $\omega$ is Gauduchon, i.e. $\partial \bar{\partial} \omega^{n-1} = 0$, the Aeppli class $[\omega^{n-1}]_A$ is well-defined.
	
	For a $d$-closed $(1,1)$-form $\alpha$, we have $$\partial \bar{\partial} (\omega^{n-1} \wedge \alpha)= \partial \bar{\partial} \omega^{n-1} \wedge \alpha= 0.$$
	
	If $\alpha = \partial \bar{\partial} \varphi$ for some $\varphi \in C^{\infty}(X, \mathbb{C})$, we have $$\omega^{n-1} \wedge \partial \bar{\partial} \varphi = \partial(\omega^{n-1} \wedge \bar{\partial} \varphi) + \bar{\partial}(\varphi\partial  \omega^{n-1})\in\mathrm{Im}\partial + \mathrm{Im}\bar{\partial}.$$
	
	($\Rightarrow$) If $\omega^{n-1}  = \partial \beta + \bar{\partial} \gamma$ for some $\beta\in\cinf_{n-2,n-1}(X,\cc)$, $\gamma\in\cinf_{n-1,n-2}(X,\cc)$,  we have $$\omega^{n-1} \wedge \alpha = \partial(\beta \wedge \alpha) - \bar{\partial}(\gamma \wedge \alpha)\in\mathrm{Im}\partial + \mathrm{Im}\bar{\partial}$$for all $d$-closed $(1,1)$-forms $\alpha$.
	
	Hence, for all $[\alpha]_{BC}\in\h[BC]{1,1}(X,\cc)$, we have $$[\alpha]_{BC}\wedge[\omega^{n-1}]_A=[\omega^{n-1} \wedge \alpha]_A = 0.$$

	($\Leftarrow$)Denote by $\Delta_{A}$ and $\Delta_{BC}$ the Aeppli Laplacian and Bott-Chern Laplacian  induced by $\omega$.
	
	Because we have the orthogonal 3-space decomposition \cite{schweitzer2007autour}:$$C^{\infty}_{n-1,n-1}(X, \mathbb{C}) = \ker \Delta_{A} \oplus (\mathrm{Im}\, \partial + \mathrm{Im}\, \bar{\partial}) \oplus \mathrm{Im}(\bar{\partial} \partial)^*$$ and $\ddb \omega^{n-1}=0$, there is a decomposition $$\omega^{n-1} = (\omega^{n-1})_h + \partial \Gamma + \bar{\partial} \Gamma'.$$
	
	Due to another orthogonal 3-space decomposition (\cite{5a8ecba0-bf3c-3a04-94d0-f3fe0a96f498}, see also\cite{schweitzer2007autour}): $$C^{\infty}_{1,1}(X, \mathbb{C}) = \ker \Delta_{BC} \oplus \mathrm{Im}(\partial \bar{\partial}) \oplus (\mathrm{Im}\, \partial^* + \mathrm{Im}\, \bar{\partial}^*),$$ for all $d$-closed $(1,1)$-forms $\alpha$, we have $$\alpha = \alpha_h + \partial \bar{\partial} \varphi$$for some $\varphi\in\cinf(X,\cc)$.
	
	$H^{1,1}_{BC}(X, \mathbb{C}) \wedge [\omega^{n-1}]_A = 0$ means for every $\alpha_h \in \text{ker} \Delta_{BC}$, we have $$\omega^{n-1} \wedge \alpha_h = (\omega^{n-1})_h \wedge \alpha_h + \partial(\alpha_h \wedge \varphi') + \bar{\partial}(\alpha_h \wedge \psi')\in\mathrm{Im}\partial + \mathrm{Im}\bar{\partial}.$$
	
	The duality of the two decompositions mentioned above implies $$\star(\omega^{n-1})_h \in \text{ker} \Delta_{BC}.$$ Therefore, we have $$(\omega^{n-1})_h \wedge \star (\omega^{n-1})_h = |(\omega^{n-1})_h|^2 dV\in \mathrm{Im}\partial + \mathrm{Im}\bar{\partial}.$$
	Hence we have $$\int_X |(\omega^{n-1})_h|^2 dV = 0.$$
	Thus, we deduce $(\omega^{n-1})_h = 0$ and $\omega^{n-1} \in \mathrm{Im}\partial + \mathrm{Im}\bar{\partial}.$
	
	(3)($\Rightarrow$) 	Let $\Omega$ be a form as in (1). Suppose there is a non-zero $d$-closed $(1,1)$-current $T \geq 0$. Because $\omega^{n-1}$ is positive definite, we have $$\int_X T \wedge \Omega = \int_X T \wedge \omega^{n-1} > 0.$$
	
	On the other hand, by the $d$-closedness of $T$ and the $d$-exactness of $\Omega$, we know that $T \wedge \Omega$ is $d$-exact. Therefore, we have $$\int_X T \wedge \Omega = 0.$$ This would be a contradiction.
	
	($\Leftarrow$)Let $\mathcal{E}'_2(X)$ (resp. $\mathcal{E}'_{1,1}(X)$) be the space of currents of dimension 2 (resp. bidimension $(1,1)$), and let $\mathcal{A}$ be the convex closed subspace of $\mathcal{E}'_2(X)$ of $d$-closed currents of dimension 2. 
	
	Fix a Hermitian metric $\omega$ on $X$. We denote $\mathcal{B} = \{T \in \mathcal{E}'_{1,1}(X) \mid \int_X T \wedge \omega^{n-1} = 1\}$. Then $\mathcal{B}$ is a convex compact subset of $\mathcal{E}'_2(X)$ by \cite{sullivan1976cycles}.
	
	Suppose if there exists no non-zero $(1,1)$-current $T \ge 0$, i.e. $\mathcal{A} \cap \mathcal{B} = \varnothing$. By the Hahn-Banach separation theorem, there exists a linear functional that vanishes identically on $\mathcal{A}$ and is positive on $\mathcal{B}$. That is to say, there exists a $d$-exact form $\Omega$ of degree 2 whose $(1,1)$-component is a Hermitian metric.
\end{proof}
Due to the compactness of $X$, the $d$-exact $(2n-2)$-form $\Omega$ in Theorem \ref{degsg} (1)  is $\tilde{d}$(bounded). Hence, it is clear that every degenerate sG metric is sG-hyperbolic.
\begin{corollary}
	If a compact complex manifold $X$ is degenerate sG, $X$ is divisorially hyperbolic.
\end{corollary}

We have the deformation openness of the degenerate sG condition.
\begin{theorem}
	Let \(\pi: X \rightarrow B\) be a holomorphic family of compact complex manifolds \(X_t := \pi^{-1}(t)\), with \(t \in B\). Fix an arbitrary reference point \(0 \in B\). If the fibre \(X_0\) is a degenerate sG manifold, then, for all \(t \in B\) sufficiently close to 0, the fibre \(X_t\) is again a degenerate sG manifold.
	
\end{theorem}
\begin{proof}
	This is quite obvious because of Theorem \ref{degsg}(1). Deformation does not change the $d$-exactness of $\Omega$ because it does not change the differentiable structure of the fibre $X_0$. The \((n-1,n-1)\)-component \(\Omega^{n-1,n-1}_t\) of $\Omega$ with respect to the complex structure of \(X_t\) is positive definite for \(t \in B\) sufficiently close to \(0\) due to the continuity of \(\Omega^{n-1,n-1}_t\) with respect to $t$. By Lemma (\cite{10.1007/BF02392356}, (4.8)), there exists a metric \(\omega_t\) such that \(\omega_t^{n-1} = \Omega_t^{n-1, n-1}\).
%
\end{proof}

(b)We now present a more concrete example.

Let \( G \) be a semi-simple complex Lie group, and \( \Gamma \) be a co-compact lattice of \( G \). By \cite{dan21}, \( G/\Gamma \) is balanced hyperbolic. An even stronger statement holds: it is actually degenerate balanced (see \cite{popovici2013aeppli}) by \cite{yachou1998varietes}.

We have the deformation openness of sG-hyperbolicity, but not of the balanced condition by \cite{alessandrini1990small}. Now we take \( G = SL_2(\mathbb{C}) \), $\Gamma$ a co-compact lattice of $G$. The deformations of \( X:=G/\Gamma \) are sG-hyperbolic but not necessarily balanced hyperbolic. We basically follow the process of \cite{rajan1994deformations}.

We choose a co-compact lattice $\Gamma$ of non-zero first Betti number. Let \( K \) be a maximal compact subgroup of \( G \) with an invariant Hermitian metric on \( G \). Fix a maximal torus \( S \) of \( K \) and a system of positive roots of \( G \) with respect to \( S \). Because the first Betti number of the lattice $\Gamma$ is not zero, \( H^{0,1}(X, T^{1,0}X) \) is not zero. Let \( \lambda \) be a highest weight of \( K \) on \( H^{0,1}(X, T^{1,0}X) \), and \( V \) be the corresponding highest weight subspace. By Theorem 3 of \cite{rajan1994deformations}, \( G/\Gamma \) can be deformed in all directions in \( V \).

(c)By Proposition \ref{product}, we have the following
\begin{corollary}
	If $X_1$ is a degenerate sG manifold and $X_2$ a balanced hyperbolic manifold, $X_1\times X_2$ is an sG-hyperbolic manifold.
\end{corollary}
But $X_1\times X_2$ is not necessarily a degenerate sG or balanced hyperbolic manifold.

\section{ Weakly $p$-K\"ahler Hyperbolicity and  Pluriclosed Star Split Hyperbolicity}

Building on our exploration of sG-hyperbolic manifolds and their exemplifications, we now expand our horizon to encompass two other hyperbolicity variants: weakly p-Kähler hyperbolicity and pluriclosed star split hyperbolicity.

Recall that a $(p,p)$-form  $\Omega$  is called a weakly $p$-K\"ahler structure if  $\Omega$  is weakly strictly positive and  is the $(p,p)$-component of a real $d$-closed $2p$-form $\hat{\Omega}$.

\begin{definition}\label{defp}
	Let  $X$  be a compact complex manifold with 
	$\operatorname{dim}_{\cc} X=n \geqslant 2$. A weakly $p$-Kähler structure $\Omega$ is said
	to be {\bf weakly $p$-K\"ahler hyperbolic} if there exists a $d$-closed $2p$-form  $\hat{\Omega}$ , such that  $\Omega$  is the $(p,p)$-component of  $\hat{\Omega}$, and  $\hat{\Omega}$  is  $\tilde{d}$(bounded) with respect to an arbitrary metric  $\omega$  on  $X$.
\end{definition}

For two metrics  $\omega_{1}$  and  $\omega_{2}$  on  X , we have  $\frac{1}{C} \omega_{1}\le \omega_{2}\le C \omega_{1}$ , for some constant  $C>0$  because of the compactness of  $X$. Thus, we deduce that the property of subexponential growth of a function is independent of the choice of metric, and that Definition \ref{defp} is well-posed.

\begin{theorem}
	Let  $X$ be a compact complex  manifold of dimension  $n$. If $X$  is weakly $p$-K\"ahler hyperbolic, then  there is no holomorphic map $f: \mathbb{C}^{p} \to X$ such that $f$ is non-degenerate at some point and has subexponential growth (with $\operatorname{Vol}_{\Omega, f}\left(B_{t}\right):=\int_{B_{t}} f^{*} \Omega$).
\end{theorem}
\begin{proof}
	Suppose there exists a holomorphic map $f:\cc^{p}\rightarrow X$ non-degenerate at some point and has subexponential growth.

	Let $\pi: \tilde{X} \longrightarrow X$ be the universal cover of $X$. Fix a metric $\omega$ on $X$. There exists a  $\pi^{*}\omega$-bounded  $(2p-1)$-form  $\Gamma$  on  $\tilde{X}$, such that  $d \Gamma=\pi^{*} \hat\Omega$ .
	
	Then $\tilde{f}^{*} \Gamma$ is $f^*\omega$-bounded:
	\begin{align*}
		\left|\tilde{f}^{*} \Gamma\left(v_{1}, \cdots v_{2 p-1}\right)\right| & =\left|\Gamma\left(f_{*} v_{1}, \cdots, f_{*} v_{2 p-1}\right)\right| \\
		& \leqslant C\left|\tilde{f}_{*} v_{1}\right|_{\pi^{*} \omega} \cdots\left|\tilde f_{*} v_{2 p-1}\right|_{\pi^{*} \omega} \\
		& =C\left|v_{1}\right|_{f^{*} \omega} \cdots\left|v_{2 p-1}\right|_{f^{*} \omega}
	\end{align*}
	for any tangent vectors $v_{1}, \cdots, v_{2 p-1}$ in $\mathbb{C}^{p}$.
	
	Now, we have
	$$
	\begin{aligned}
		0<\langle f^*\Omega,[B_r]\rangle=\text { Vol }_{\omega, f}\left(B_{r}\right) & =\int_{B r} f^{*} \hat\Omega \\
		& =\int_{B r} d\left(\tilde{f}^{*}\Gamma\right) \leqslant C \int_{S_{r}} d \sigma_{\omega, f, r},
	\end{aligned}$$
	
	where  $d \sigma_{\omega, f,r}=\left.{ }({\star}_{f^{*} \omega}\left(\frac{d \tau}{|d \tau|_{f^{*} \omega}{ }}\right)\right)|_{s_{r}}$. 
	
	By H\"older inequality, we have
	
	$$\int_{s_{r}} \frac{1}{|d \tau|_{f^{*} \omega}} d \sigma_{\omega, f, r} \cdot \int_{S_{r}} |d \tau|_{f^{*} w} d \sigma_{\omega, f, r} \geqslant\left(\int_{S_{r}} d \sigma_{\omega, f,r}\right)^{2}.$$
	
	We have  $d \tau=2 r d r$ . Let  $d \mu_{\omega, f, r}$  be the measure on  $S_{r}$  such that
	
	$$d \mu_{\omega, f, r} \wedge \frac{(d \tau)|_{S_r}}{2 r}=\left.\left(f^{*} \omega_{n-1}\right)\right|_{S_r} \text {. }$$
	
	Hence we have  $$\frac{1}{2 r} d \mu_{\omega,f , r}=\frac{1}{|d \tau|_{f^{*}\omega} } d \sigma_{\omega, f, r}. $$
	Then we get:
	$$\begin{array}{l}
		\text {Vol}_{\omega, f}\left(B_{r}\right)=\int_{0}^{r}\left(\int_{S_t} \frac{1}{| d \tau |_{f^*\omega} } d \sigma_{\omega, f, t}\right) d \tau \\
		\geqslant \int_{0}^{r} \frac{\left(\int_{S_{r}} d \sigma_{\omega, f,r}\right)^{2}}{\int_{S_t}\left|d \tau\right|_{f^{*} \omega} d \sigma_{\omega, f, t}} 2 t d t \\
		
		\geqslant \frac{2}{C^{2}} \int_{0}^{r} \frac{\left(\text {Vol}_{\omega, f}\left(B_{t}\right)\right)^{2}}{\int_{S_{t}}|d \tau|_{f^{*} \omega} d{\sigma_{\omega, t}}}td t \\
		\geqslant \frac{2}{C_{1} C^{2}} \int_{0}^{r} \text {Vol}_{\omega, f}\left(B_{t}\right) d t
	\end{array}$$
	for $r$ big enough.
	
	This is to say
	$$
	F^{\prime}(r) \geqslant \frac{2}{C_{1} C^{2}} F(r)$$
	for $r$ big enough. 
	
	Hence, we get$$	(\log F(r))^{\prime} \geqslant \frac{2}{C_{1} C^{2}}.$$
	Finally, we get	$$\log F\left(r_{1}\right)-\frac{2}{C_{1} C^{2}} r_{1}+\frac{2}{C_{1} C^{2}} r_{2} \geqslant \log F\left(r_{2}\right).
	$$
	By (ii) of Definition \ref{div}, we have $F(r_2)=0$ for $r_2$ big enough, which contradicts our assumption.
\end{proof}

\begin{theorem}
	Let $X$ be a compact complex manifold with $\dim_{\mathbb{C}}X=n$. Let $\pi:\tilde{X}\rightarrow X$ be the universal cover of $X$.
	If $X$ is weakly p-K\"ahler hyperbolic, then there exists no
	non-zero $d$-closed positive (n-p,n-p)-current $\tilde{T}\ge 0$
	on $\tilde{X}$ such that $\tilde{T}$ is of $L^1_{\pi^*\omega}$.
\end{theorem}

\begin{proof}
	Let $\hat\Omega$ be a \( d \)-closed \( 2p \)-form as in the definition of weakly p-Kähler hyperbolic manifold. There exists a \( L^\infty_{\pi^*\omega} \)-form \( \Gamma \) of degree \( (2p-1) \) on \( \tilde{X} \) such that \( \pi^*\hat\Omega = d\Gamma \).
	
	For a \( d \)-closed \( (n-p,n-p) \)-current \( \tilde{T}\) of \( L^1_{\pi^*\omega} \) on \( \tilde{X} \), \( \tilde{T} \wedge \Gamma \) is again  of \( L^1_{\pi^*\omega} \). Hence also \( d(\tilde{T} \wedge \Gamma) \). Now we have:
	
	\[
	\int_{\tilde{X}} \tilde{T} \wedge \pi^*\hat\Omega = \int_{\tilde{X}} d(\tilde{T} \wedge \Gamma) = 0.
	\]
	
	However, $(p,p)$-component of $\pi^*\hat\Omega$  is weakly strictly positive. Therefore, there exists no non-zero $d$-closed positive $(n-p,n-p)$-current $\tilde{T}\ge 0$
	on $\tilde{X}$ such that $\tilde{T}$ is of $L^1_{\pi^*\omega}$.
	
\end{proof}

Recall the definition of pluriclosed star split metric:

\begin{definition}\cite{popovici2023pluriclosed}
	Let  $X$  be a complex manifold with  $dim_{\cc} X=n$ and  $\omega$ a Hermitian  metric   on  $X$. Let $\star$ be the Hodge star operator induced by $\omega$ and $\rho_{\omega}$ the unique $(1,1)$-form such that $i\ddb\omega_{n-2}=\omega_{n- 2}\wedge\rho_{\omega}$. The metric $\omega$ is said to be pluriclosed star split if $\ddb(\star\rho_{\omega})=0$.
\end{definition}


\begin{definition}
	Let  $X$  be a compact complex manifold with  $dim_{\cc} X=n$. A metric  $\omega$  on  $X$  is said to be {\bf pluriclosed star split hyperbolic } if  $\omega$  is pluriclosed star split  $\pi^{*}\left(\star\rho_{\omega}\right)=\partial \bar{\Gamma}+\bar{\partial} \Gamma$  on  $\tilde{X}$  with  $\Gamma$ $\tilde{\omega}$-bounded.
\end{definition}

\begin{theorem}
	Let  $X$ be a compact complex  manifold of dimension  $n$. If $X$  is pluriclosed star split hyperbolic, then  $X$  is divisorially hyperbolic.
\end{theorem}

\begin{proof}
	Suppose there exists a holomorphic map $f:\cc^{n-1}\rightarrow X$ non-degenerate at some point and has subexponential growth.

	Let $\pi: \tilde{X} \longrightarrow X$ be the universal cover of $X$. There exists a  $\pi^{*}\omega$-bounded  $(2n-3)$-form  $\Gamma$  on  $\tilde{X}$, such that  $\pi^{*}\left(\omega_{n-1}\right)=\partial \bar{\Gamma}+\bar{\partial} \Gamma$.
	
	Then $\tilde{f}^{*} \Gamma$ is $f^*\omega$-bounded:
	\begin{align*}
		\left|\tilde{f}^{*} \Gamma\left(v_{1}, \cdots v_{2 n-3}\right)\right| & =\left|\Gamma\left(f_{*} v_{1}, \cdots, f_{*} v_{2 n-3}\right)\right| \\
		& \leqslant C\left|\tilde{f}_{*} v_{1}\right|_{\pi^{*} \omega} \cdots\left|\tilde f_{*} v_{2 n-3}\right|_{\pi^{*} \omega} \\
		& =C\left|v_{1}\right|_{f^{*} \omega} \cdots\left|v_{2 n-3}\right|_{f^{*} \omega}
	\end{align*}
	for any tangent vectors $v_{1}, \cdots, v_{2 n-3}$ in $\mathbb{C}^{n-1}$.
	
	Because of the compactness of $X$, $\frac{i\bar\partial \omega_{n- 2}}{g}$ is bounded. According to \cite{popovici2023pluriclosed}, we have $\star\rho_{\omega}=g\omega_{n-1}-i\partial \bar{\partial} \omega_{n-2}$ for some real-valued $\cinf$ function on $X$. Therefore, we have
	$$
	\begin{aligned}
		\text { Vol }_{\omega, f}\left(B_{r}\right) & =\int_{B_{r}} f^{*}\left(\omega_{n-1}\right) \\
		&=(n-1) \int_{B_{r}} f^{*}\left(\frac{1}{g} \star\rho_{\omega}+\frac{i}{g} \partial \bar{\partial} \omega_{n-2}\right) \\
		&=(n-1) \int_{B_{r}} d\left({\tilde{f}}^{*}\left(\frac{\Gamma+\bar{\Gamma}}{g}\right)+f^*\left(\frac{i\bar\partial \omega_{n- 2}}{g}\right)\right) \\&\leqslant C \int_{S_{r}} d \sigma_{\omega, f, r},
	\end{aligned}$$

	where  $d \sigma_{\omega, f,r}=\left.{ }({\star}_{f^{*} \omega}\left(\frac{d \tau}{|d \tau|_{f^{*} \omega}{ }}\right)\right)|_{s_{r}}$.
	
	By H\"older inequality, we have
	
	$$\int_{s_{r}} \frac{1}{|d \tau|_{f^{*} \omega}} d \sigma_{\omega, f, r} \cdot \int_{S_{r}} |d \tau|_{f^{*} w} d \sigma_{\omega, f, r} \geqslant\left(\int_{S_{r}} d \sigma_{\omega, f,r}\right)^{2}.$$
	
	We have  $d \tau=2 r d r$ . Let  $d \mu_{\omega, f, r}$  be the measure on  $S_{r}$  such that
	
	$$d \mu_{\omega, f, r} \wedge \frac{(d \tau)|_{S_r}}{2 r}=\left.\left(f^{*} \omega_{n-1}\right)\right|_{S_r} \text {. }$$
	
	Hence, we have  $$\frac{1}{2 r} d \mu_{\omega,f , r}=\frac{1}{|d \tau|_{f^{*}\omega} } d \sigma_{\omega, f, r}.$$
	
	Then we get:
	$$\begin{array}{l}
		\text {Vol}_{\omega, f}\left(B_{r}\right)=\int_{0}^{r}\left(\int_{S_t} \frac{1}{| d \tau |_{f^*\omega} } d \sigma_{\omega, f, t}\right) d \tau \\
		\geqslant \int_{0}^{r} \frac{\left(\int_{S_{r}} d \sigma_{\omega, f,r}\right)^{2}}{\int_{S_t}\left|d \tau\right|_{f^{*} \omega} d \sigma_{\omega, f, t}} 2 t d t \\
		
		\geqslant \frac{2}{C^{2}} \int_{0}^{r} \frac{\left(\text {Vol}_{\omega, f}\left(B_{t}\right)\right)^{2}}{\int_{S_{t}}|d \tau|_{f^{*} \omega} d{\sigma_{\omega, t}}}td t \\
		\geqslant \frac{2}{C_{1} C^{2}} \int_{0}^{r} \text {Vol}_{\omega, f}\left(B_{t}\right) d t
	\end{array}$$
	for $r$ big enough.
	
	This is to say
	$$
		F^{\prime}(r) \geqslant \frac{2}{C_{1} C^{2}} F(r) $$ 
		for $r$ big enough.
		
	Hence we get$$	(\log F(r))^{\prime} \geqslant \frac{2}{C_{1} C^{2}}.$$
	Finally, we get	$$\log F\left(r_{1}\right)-\frac{2}{C_{1} C^{2}} r_{1}+\frac{2}{C_{1} C^{2}} r_{2} \geqslant \log F\left(r_{2}\right).
	$$
	By (ii) of Definition \ref{div}, we have $F(r_2)=0$ for $r_2$ big enough, which contradicts our assumption.
\end{proof}
\bibliographystyle{alpha}
\bibliography{bib,cit}	
	\vspace*{2em}
	Institut de Math\'ematiques de Toulouse,\\
	Universit\'e Paul Sabatier,\\
	118 route de Narbonne, 31062 Toulouse, France\\
	Email: yi.ma@math.univ-toulouse.fr
\end{document}